\documentclass[12pt]{article}

\usepackage[T1]{fontenc}
\usepackage[english]{babel}
\usepackage{amsmath,amssymb,amsthm,mathrsfs}
\usepackage[left=2cm,right=2cm,top=2cm,bottom=2cm]{geometry}
\usepackage{longtable,array}
\DeclareMathOperator{\tr}{tr}
\DeclareMathOperator{\sgn}{sgn}
\newcommand{\dd}{\mathop{}\!\mathrm d}
\newcommand{\ii}{\mathrm i}
\newcommand{\Id}{\operatorname{Id}}
\newcommand{\ord}{\operatorname{ord}}

\theoremstyle{plain}
\newtheorem{theorem}{Theorem}

\newtheorem{proposition}{Proposition}
\newtheorem{definition}{Definition}
\theoremstyle{definition}
\newtheorem{remark}{Remark}

\title{From quadratic integrals to Nijenhuis operators: two-dimensional dictionary}
\author{Dinmukhammed Akpan\footnote{Institut f\"ur Mathematik, Friedrich Schiller Universit\"at Jena, 07743 Jena, Germany; Institute of Mathematics and Mathematical Modeling, Almaty, Kazakhstan; \texttt{dinmukhammed.akpan@uni-jena.de}.}}
\date{}

\begin{document}
\maketitle

\begin{abstract}
We construct an explicit dictionary between two local classifications in dimension two: normal forms of pseudo-Riemannian metrics with a quadratic integral of the geodesic flow and normal forms of $\operatorname{gl}$-regular Nijenhuis operators. The correspondence is obtained by an explicit change of coordinates constructed from the trace and determinant of the operator. 
\end{abstract}

\section{Introduction and preliminaries}

Let $g$ be a pseudo-Riemannian metric of signature $(+,-)$ on a two-dimensional manifold $M^2$. Its geodesic flow is the
Hamiltonian system on $T^*M$ with Hamiltonian
\[
\mathcal H=\frac12 g^{ij}p_ip_j.
\]
We assume that the geodesic flow admits an integral quadratic in the momenta,
\[
F=a(x,y)p_x^2+b(x,y)p_xp_y+c(x,y)p_y^2,
\qquad \{\mathcal H,F\}=0.
\]

Local normal forms of two-dimensional pseudo-Riemannian metrics $g$ admitting a quadratic integral $F$ of the geodesic flow were obtained in \cite{AkpanBolsinov, BMP}. Any metric and a quadratic integral of its geodesic flow determine a $(1,1)$-tensor field
\begin{equation}\label{eq:L-definition} L:=\operatorname{adj}(g)\operatorname{adj}(F) =(\det g)(\det F)g^{-1}F^{-1}. \end{equation}
An important property of this construction is that $L$ is a Nijenhuis operator. The cases B1.1--B4 from \cite{AkpanBolsinov} give $\operatorname{gl}$-regular Nijenhuis operators. On the other hand, all local normal forms of two-dimensional $\operatorname{gl}$-regular Nijenhuis operators were classified independently in \cite{BKM-NGIII}. The two classifications use different coordinate systems and different sets of parameters, so the normal form of $L$ cannot be read directly from the displayed formulas for $(g,F)$. The purpose of this paper is to provide a dictionary between these two lists.

Every two-dimensional pseudo-Riemannian metric admits a local null coordinate system in which
\[
g=2\lambda(x,y)\,\dd x\,\dd y,
\qquad \mathcal H = \frac{p_xp_y}{\lambda(x,y)}, \quad \lambda(0,0)\neq0
\]
and the condition $\{\mathcal H,F\}=0$ is equivalent to
\begin{equation}\label{eq:Killing-system}
\begin{aligned}
a_y&=0, & c_x&=0,\\
\lambda a_x+2\lambda_xa+(\lambda b)_y&=0,
&
\lambda c_y+2\lambda_yc+(\lambda b)_x&=0.
\end{aligned}
\end{equation}
Thus $a=a(x)$ and $c=c(y)$. A direct calculation in \eqref{eq:L-definition} gives
\begin{equation}\label{eq:L-null}
L=
\begin{pmatrix}
\frac{\lambda b}{2} & -a\lambda\\
-c\lambda & \frac{\lambda b}{2}
\end{pmatrix},
\qquad
\tr L=\lambda b,
\qquad
\det L=\frac{\lambda^2}{4}(b^2-4ac).
\end{equation}

\begin{proposition}\label{prop:bridge}
If $F$ is a quadratic integral of the geodesic flow of $g$ and $L$ is
defined by \eqref{eq:L-definition}, then $L$ is a Nijenhuis operator; that
is, for all vector fields $\xi$ and $\zeta$,
\[
N_L(\xi,\zeta):=\underbrace{L^2[\xi,\zeta]+[L\xi,L\zeta]
-L[L\xi,\zeta]-L[\xi,L\zeta]}_{\text{Nijenhuis torsion}}\equiv0.
\]
\end{proposition}

\begin{proof}
The vanishing of the Nijenhuis torsion of a $(1,1)$-tensor field is a
coordinate-invariant condition. We verify it in null coordinates. In dimension two, the skew-symmetry of the Nijenhuis torsion in its lower
indices implies that it has only two potentially nonzero components,
$(N_L)^1_{12}$ and $(N_L)^2_{12}$. For the operator \eqref{eq:L-null}, a
direct computation, using $a_y=0$ and $c_x=0$, gives
\[
\begin{aligned}
(N_L)^1_{12}
=a\lambda\bigl(\lambda c_y+2\lambda_yc+(\lambda b)_x\bigr),
\quad 
(N_L)^2_{12}
=-c\lambda\bigl(\lambda a_x+2\lambda_xa+(\lambda b)_y\bigr).
\end{aligned}
\]
The expressions in parentheses vanish by \eqref{eq:Killing-system}.
Consequently, both components of the Nijenhuis torsion vanish, and hence
$N_L=0$.
\end{proof}

\begin{definition}
Let $g$ be a metric defined near a point $p_0$, let $F$ be a quadratic integral
of its geodesic flow, and let $L$ be given by \eqref{eq:L-definition}. The
point $p_0$ is called \emph{generic} if the algebraic type of $L$ is constant
in a neighborhood of $p_0$, and \emph{singular} otherwise.

The operator $L(p_0)$ is called $\operatorname{gl}$-\emph{regular} if its
Jordan normal form contains exactly one Jordan block for each eigenvalue.
\end{definition}

\begin{remark}
    Without loss of generality, we shall henceforth assume that the singular point $p_0$ coincides with the origin, i.e. $p_0 = (0,0)$.
\end{remark}

\subsection{Normal forms of $(g,F)$}\label{subsec:gF}

We first fix notation for the normal forms of $(g,F)$.
At a generic point, the classification of \cite{BMP} consists of the
following three normal forms:
\begin{equation}\label{eq:R-forms}
\renewcommand{\arraystretch}{1.5}
\begin{array}{c|c|c}
\mathrm{R1}\ \text{(Liouville)}
&
\mathrm{R2}\ \text{(Complex-Liouville)}
&
\mathrm{R3}\ \text{(Jordan block)}
\\
\hline
\displaystyle
g=(X(x)-Y(y))(\dd x^2-\dd y^2)
&
\displaystyle
g=\operatorname{Im}(h)\,\dd x\,\dd y
&
\displaystyle
g=(1+xY'(y))\,\dd x\,\dd y
\\
\hline
\displaystyle
F=\frac{X(x)p_y^2-Y(y)p_x^2}{X(x)-Y(y)}
&
\displaystyle
\begin{gathered}
F=p_x^2-p_y^2+
2\frac{\operatorname{Re}(h)}
       {\operatorname{Im}(h)}p_xp_y,\\[-2pt]
\end{gathered}
&
\displaystyle
F=p_x^2-
2\frac{Y(y)}{1+xY'(y)}p_xp_y
\end{array}
\end{equation}
where $X(x)\neq Y(y)$ in R1 and $h=h(x+\ii y)$ is holomorphic, $\operatorname{Im}(h)\neq0$ in R2.

We now turn to singular points at which the associated operator $L$ is
$\operatorname{gl}$-regular. The complete list of local normal forms of
the corresponding pairs $(g,F)$ was obtained in
\cite[Theorem~2]{AkpanBolsinov} and consists of the families B1.1--B4.

Except for B1.1, all these families are obtained from the formulas
below by choosing the function $f$ and the corresponding solution $\rho$ as indicated in
Table~\ref{tab:B-families}.

Let $\rho(x,y)$ be the local solution of
\begin{equation}\label{eq:rho-flow}
\rho_x=f(\rho),
\qquad
\rho(0,y)=y.
\end{equation}

The families B1.2, B1.3, and B2 are given by
\begin{equation}\label{eq:B12-B2}
\begin{aligned}
g&=
\frac{X(\rho(x,y))-Y(\rho(-x,y))}{f(y)}
\,\dd x\,\dd y,\\
F&=
p_x^2+f(y)^2p_y^2
-2f(y)
\frac{X(\rho(x,y))+Y(\rho(-x,y))}
     {X(\rho(x,y))-Y(\rho(-x,y))}
p_xp_y, \\
X(t)&=f(t)h(t)+H(t),
 \quad Y(t)=-f(t)h(t)+H(t),
\end{aligned}
\end{equation}
where  $h$ and $H$ are
real-analytic functions with $h(0)\neq0$.

The families B3.1, B3.2, and B4 are given by
\begin{equation}\label{eq:B3-B4}
\begin{aligned}
g&=
2\frac{\operatorname{Im}W(\rho(\ii x,y))}{f(y)}
\,\dd x\,\dd y,\\
F&=
-p_x^2+f(y)^2p_y^2
+2f(y)
\frac{\operatorname{Re}W(\rho(\ii x,y))}
     {\operatorname{Im}W(\rho(\ii x,y))}
p_xp_y,\\
W(t)&=H(t)+\ii f(t)h(t),\\
\end{aligned}
\end{equation}
where  $h$ and $H$ are
real-analytic functions with $h(0)\neq0$.

With the convention $g=2\lambda\dd x\dd y,$
the functions $h$ and $H$ appearing above are the initial data
\begin{equation}\label{eq:initial-data}
\lambda(0,y)=h(y),
\qquad
\lambda_x(0,y)=H'(y).
\end{equation}

The complete list of choices of $f$ and $\rho$ is summarized in
Table~\ref{tab:B-families}. We also record the order
$m=\ord_0(f^2)$.

\begingroup
\small
\renewcommand{\arraystretch}{1.2}
\begin{longtable}{
>{\centering\arraybackslash}p{0.09\textwidth}|
>{\centering\arraybackslash}p{0.16\textwidth}|
>{\centering\arraybackslash}p{0.24\textwidth}|
>{\centering\arraybackslash}p{0.31\textwidth}|
>{\centering\arraybackslash}p{0.07\textwidth}}
\caption{}
\label{tab:B-families}\\

\hline
\textbf{Family}
&
\textbf{Formula}
&
\textbf{$f(y)$}
&
\textbf{$\rho(x,y)$}
&
\textbf{$m$}
\\
\hline
\endfirsthead

\hline
\textbf{Family}
&
\textbf{Formula}
&
\textbf{$f(y)$}
&
\textbf{$\rho(x,y)$}
&
\textbf{$m$}
\\
\hline
\endhead

\hline
\endfoot

B1.2
&
\eqref{eq:B12-B2}
&
$y$
&
$ye^x$
&
$2$
\\
\hline

B1.3
&
\eqref{eq:B12-B2}
&
$\displaystyle\frac{2}{2-s}y^{s/2}$, $s>2$
&
$\displaystyle
y\left(1+xy^{(s-2)/2}\right)^{2/(2-s)}$
&
$s$
\\
\hline

B2
&
\eqref{eq:B12-B2}
&
$\displaystyle\frac{y^k}{1+y^{k-1}}$, $k\geq2$
&
the solution of \eqref{eq:rho-flow}
&
$2k$
\\
\hline

B3.1
&
\eqref{eq:B3-B4}
&
$y$
&
$ye^x$
&
$2$
\\
\hline

B3.2
&
\eqref{eq:B3-B4}
&
$\displaystyle\frac{y^k}{1-k}$, $k\geq2$
&
$\displaystyle
y\left(1+xy^{k-1}\right)^{1/(1-k)}$
&
$2k$
\\
\hline

B4
&
\eqref{eq:B3-B4}
&
$\displaystyle\frac{y^k}{1+y^{k-1}}$, $k\geq2$
&
the solution of \eqref{eq:rho-flow}
&
$2k$
\\
\hline

\end{longtable}
\endgroup

The remaining family B1.1 is written separately:
\begin{equation}\label{eq:B11}
\begin{aligned}
g&=
\frac{X(x/2+\sqrt y)-X(x/2-\sqrt y)}{\sqrt y}
\,\dd x\,\dd y,\\
F&=
p_x^2+yp_y^2
-2\sqrt y\,
\frac{X(x/2+\sqrt y)+X(x/2-\sqrt y)}
     {X(x/2+\sqrt y)-X(x/2-\sqrt y)}
p_xp_y,
\end{aligned}
\end{equation}
where $X$ is analytic and $X'(0)\neq0$.

\subsection{Notation for the Nijenhuis normal forms}

We use $(x,y)$ for the coordinates in which the metric--integral pair
$(g,F)$ is written and $(\mathsf{x},\mathsf{y})$ for the coordinates of the
corresponding Nijenhuis operator. These coordinate systems are, in
general, different. At a generic point every two-dimensional
$\operatorname{gl}$-regular Nijenhuis operator can be brought to one of the
following local normal forms:
\begin{equation}\label{eq:G-forms}
\begin{aligned}
\mathrm{G1}:&\quad
L=\begin{pmatrix}
\phi(\mathsf{x})&0\\
0&\psi(\mathsf{y})
\end{pmatrix},
&&\phi(\mathsf{x})\neq \psi(\mathsf{y}),\\
\mathrm{G2}:&\quad
L=\begin{pmatrix}
\phi(\mathsf{x},\mathsf{y})&-\psi(\mathsf{x},\mathsf{y})\\
\psi(\mathsf{x},\mathsf{y})&\phi(\mathsf{x},\mathsf{y})
\end{pmatrix},
&&\phi+\ii\psi\text{ is holomorphic in }\mathsf{x}+\ii \mathsf{y},\quad \psi\neq0,\\
\mathrm{G3}:&\quad
L=\begin{pmatrix}
\phi(\mathsf{y})&1\\
0&\phi(\mathsf{y})
\end{pmatrix}.
\end{aligned}
\end{equation}

\begin{remark}\label{rem:normalization}
At a singular point $p_0$ in dimension two, the eigenvalues of the
operator $L(p_0)$ coincide. Under the $\operatorname{gl}$-regularity
assumption, $L(p_0)$ consists of a single Jordan block. Hence, after an
appropriate linear change of coordinates,
\[
L(p_0)=\lambda_0\Id+J_0,
\qquad
J_0=
\begin{pmatrix}
0&1\\
0&0
\end{pmatrix},
\qquad
\lambda_0=\frac12\tr L(p_0).
\]
Since $L-\lambda_0\Id$ is again a Nijenhuis operator, we shall work
without loss of generality with the normalized representative for which
\[
L(p_0)=J_0.
\]

The classification below is carried out under this normalization.
\end{remark}

The complete list of local normal forms under this assumption was obtained
in \cite[Theorem~5.1]{BKM-NGIII}. We recall the part of this list needed
below.

The elementary series are
\begin{equation}\label{eq:LMN}
\begin{gathered}
L_{\mathrm{nil}}=
\begin{pmatrix}
0&1\\
0&0
\end{pmatrix},
\qquad
L_{\mathrm{nd}}=
\begin{pmatrix}
\mathsf{x}&1\\
\mathsf{y}&0
\end{pmatrix},
\qquad
M_{2\mathsf{k}-1}=
\begin{pmatrix}
0&1\\
0&\mathsf{y}^{2\mathsf{k}-1}
\end{pmatrix},
\\[5pt]
M_{2\mathsf{k}}^{\varepsilon}=
\begin{pmatrix}
0&1\\
0&\varepsilon \mathsf{y}^{2\mathsf{k}}
\end{pmatrix},
\qquad
N_{2\mathsf{k}-1}=
\begin{pmatrix}
\mathsf{y}^{2\mathsf{k}-1}&1\\
0&\mathsf{y}^{2\mathsf{k}-1}
\end{pmatrix},
\qquad
N_{2\mathsf{k}}^{\varepsilon}=
\begin{pmatrix}
\varepsilon \mathsf{y}^{2\mathsf{k}}&1\\
0&\varepsilon \mathsf{y}^{2\mathsf{k}}
\end{pmatrix}.
\end{gathered}
\end{equation}
where $\mathsf{k}\geq1$ and $\varepsilon\in\{\pm1\}$.

The remaining series are described in terms of the scalar
invariants $v=\tr L,
\
u=-\det L.$ If $\dd v\wedge\dd u\not\equiv0$, then at every point where
$\dd v\wedge\dd u\neq0$ the operator can be reconstructed from these
invariants by
\begin{equation}\label{eq:reconstruction}
L=
\begin{pmatrix}
v_{\mathsf{x}}&v_{\mathsf{y}}\\
u_{\mathsf{x}}&u_{\mathsf{y}}
\end{pmatrix}^{-1}
\begin{pmatrix}
v&1\\
u&0
\end{pmatrix}
\begin{pmatrix}
v_{\mathsf{x}}&v_{\mathsf{y}}\\
u_{\mathsf{x}}&u_{\mathsf{y}}
\end{pmatrix}.
\end{equation}

Let
\[
\boldsymbol{\mathsf{c}}=(\mathsf{c}_0,\dots,\mathsf{c}_{\mathsf{k}-1})\in\mathbb R^{\mathsf{k}},
\qquad
C(\mathsf{y})=\mathsf{c}_{\mathsf{k}-1}\mathsf{y}^{\mathsf{k}-1}
+\cdots+\mathsf{c}_1\mathsf{y}+\mathsf{c}_0.
\]
The remaining normal forms are summarized in
Table~\ref{tab:OPS-series}.

\begingroup
\small
\renewcommand{\arraystretch}{1.35}
\setlength{\tabcolsep}{3pt}
\begin{longtable}{
>{\centering\arraybackslash}p{0.11\textwidth}|
>{\centering\arraybackslash}p{0.25\textwidth}|
>{\centering\arraybackslash}p{0.08\textwidth}|
>{\centering\arraybackslash}p{0.17\textwidth}|
>{\raggedright\arraybackslash}p{0.29\textwidth}}
\caption{}
\label{tab:OPS-series}\\
\hline
\textbf{Series}&
\centering\arraybackslash\textbf{$v$}&
\textbf{$u$}&
\centering\arraybackslash\textbf{$\alpha$}&
\centering\arraybackslash\textbf{Parameters}\\
\hline
\endfirsthead
\hline
\textbf{Series}&
\centering\arraybackslash\textbf{$v$}&
\textbf{$u$}&
\centering\arraybackslash\textbf{$\alpha$}&
\centering\arraybackslash\textbf{Parameters}\\
\hline
\endhead
\hline
\endfoot

$O_{\mathsf{k},\boldsymbol{\mathsf{c}}}^{\mathsf{d},\varepsilon}$&
$\alpha\mathsf{x}\mathsf{y}^{2\mathsf{k}-1}
+\mathsf{y}^{\mathsf{k}}C(\mathsf{y})$&
$\varepsilon\mathsf{y}^{\mathsf{d}}$&
$\displaystyle \mathsf{k}\mathsf{c}_0^2
\left(1-\frac{\mathsf{k}}{\mathsf{d}}\right)\neq0$&
$\begin{array}{l}
\mathsf{k}\geq1,\quad \mathsf{d}\geq2\mathsf{k}+1,\\
\varepsilon\in\{\pm1\},\\
\varepsilon=1\text{ if }\mathsf{d}\text{ is odd}
\end{array}$\\
\hline

$P_{\mathsf{s},\boldsymbol{\mathsf{c}}}^{\mathsf{k},\varepsilon}$&
$\alpha\mathsf{x}\mathsf{y}^{\mathsf{s}}
+\mathsf{y}^{\mathsf{s}-\mathsf{k}+1}C(\mathsf{y})
+2\varepsilon\mathsf{y}^{\mathsf{k}}$&
$-\mathsf{y}^{2\mathsf{k}}$&
$2\varepsilon\mathsf{k}\mathsf{c}_0\neq0$&
$\begin{array}{l}
\mathsf{k}\geq1,\quad \mathsf{s}\geq2\mathsf{k},\\
\varepsilon\in\{\pm1\}
\end{array}$\\
\hline

$S_{\boldsymbol{\mathsf{c}}}^{2\mathsf{k},\varepsilon}$&
$\alpha\mathsf{x}\mathsf{y}^{2\mathsf{k}-1}
+\mathsf{y}^{\mathsf{k}}C(\mathsf{y})$&
$\varepsilon\mathsf{y}^{2\mathsf{k}}$&
$\displaystyle\frac{\mathsf{k}}2
(\mathsf{c}_0^2+4\varepsilon)\neq0$&
$\mathsf{k}\geq1$, $\varepsilon\in\{\pm1\}$\\
\hline

$S_{\boldsymbol{\mathsf{c}}}^{2\mathsf{k}+1}$&
$\alpha\mathsf{x}\mathsf{y}^{2\mathsf{k}}
+\mathsf{y}^{\mathsf{k}+1}C(\mathsf{y})$&
$\mathsf{y}^{2\mathsf{k}+1}$&
$2\mathsf{k}+1$&
$\mathsf{k}\geq1$\\
\hline
\end{longtable}
\endgroup

Thus $\mathsf{k},\mathsf{s},$ and $\mathsf{d}$ are reserved for parameters of the
Nijenhuis normal forms, while $k$ and $s$ retain their meaning in the
metric--integral families.

\subsection{Recognition by the trace and determinant}

We shall use the following construction to identify the singular normal
forms. Let $L$ be the operator associated with $(g,F)$ by
\eqref{eq:L-definition}, and put
\[
v=\tr L,\qquad
u=-\det L,\qquad
J=\det\left(\frac{\partial(v,u)}{\partial(x,y)}\right).
\]
For each family considered below, direct calculation gives
\begin{equation}\label{eq:uJ-factorizations}
u=y^d\widetilde u(x,y),\qquad
J=y^N\widetilde J(x,y),\qquad
\widetilde u(0,0)\widetilde J(0,0)\neq0.
\end{equation}

Define
\begin{equation}\label{eq:w-general}
\mathsf{y}=y|\widetilde u(x,y)|^{1/d}.
\end{equation}
Then $u=\varepsilon \mathsf{y}^d,$
where $\varepsilon=\sgn\widetilde u(0,0)$ if $d$ is even; if $d$ is odd,
the sign is absorbed by replacing $\mathsf{y}$ with $-\mathsf{y}$. Since
$\mathsf{y}_y(0,0)\neq0$, the pair $(x,\mathsf{y})$ is a local analytic coordinate
system.

The transformation rule for Jacobians gives
\[
J
=
\det\left(\frac{\partial(v,u)}{\partial(x,\mathsf{y})}\right)
\det\left(\frac{\partial(x,\mathsf{y})}{\partial(x,y)}\right)
=
\varepsilon d \mathsf{y}^{d-1}
\left.\frac{\partial v}{\partial x}\right|_{\mathsf{y}} \mathsf{y}_y.
\]
Since $y$ and $\mathsf{y}$ differ by a nonvanishing analytic factor,
we may, after composition with the inverse coordinate change, write
\[
J=\mathsf{y}^N\widetilde J(x,\mathsf{y}),
\qquad
\widetilde J(0,0)\neq0.
\]
It follows that $N\geq d-1$. We may therefore set
\begin{equation}\label{eq:r-from-orders}
r=N-d+1\geq0.
\end{equation}
Moreover,
\[
\left.\frac{\partial v}{\partial x}\right|_{\mathsf{y}}
=\mathsf{y}^r\widetilde{v_x}(x,\mathsf{y}),\qquad
\widetilde{v_x}(x,\mathsf{y})
=\frac{\widetilde J(x,\mathsf{y})}
{\varepsilon d\,\mathsf{y}_y(x,\mathsf{y})},
\qquad \widetilde{v_x}(0,0)\neq0.
\]
Regard the restriction of $v$ to $x=0$ as a function of $\mathsf{y}$ and write
\begin{equation}\label{eq:trace-Taylor}
v(0,\mathsf{y})=\Pi_r(\mathsf{y})+O(\mathsf{y}^{r+1}),
\end{equation}
where $\Pi_r$ is its Taylor polynomial of degree at most $r$. Integrating
the preceding identity with respect to $x$ at fixed $\mathsf{y}$, we obtain
\[
v(x,\mathsf{y})-v(0,\mathsf{y})
=
\mathsf{y}^r\int_0^x\widetilde{v_x}(t,\mathsf{y})\,\dd t.
\]
On the other hand, \eqref{eq:trace-Taylor} gives
$v(0,\mathsf{y})-\Pi_r(\mathsf{y})
=\mathsf{y}^{r+1}\widetilde v(\mathsf{y})$ for an analytic function
$\widetilde v$. Hence
\[
v-\Pi_r(\mathsf{y})
=
\mathsf{y}^r\left(
\int_0^x\widetilde{v_x}(t,\mathsf{y})\,\dd t
+\mathsf{y}\widetilde v(\mathsf{y})
\right).
\]

Let $\alpha\neq0$ and define
\begin{equation}\label{eq:z-general}
\mathsf{x}=\frac{v-\Pi_r(\mathsf{y})}{\alpha \mathsf{y}^r}.
\end{equation}
Then $\mathsf{x}$ is analytic and $\mathsf{x}_x(0,0)=\frac{\widetilde{v_x}(0,0)}{\alpha}\neq0.$

Thus $(\mathsf{x},\mathsf{y})$ is a local analytic coordinate system, and
\begin{equation}\label{eq:normalized-vu}
u=\varepsilon \mathsf{y}^d,\qquad
v=\Pi_r(\mathsf{y})+\alpha\mathsf{x}\mathsf{y}^r.
\end{equation}

Consequently, the normal form is determined by the recognition data
$d,r,\varepsilon$, and $\Pi_r$. Below, $\alpha$ is chosen as prescribed
by the corresponding normal form.

We can now state the correspondence between the two classifications.

\section{Dictionary}
For the normal forms B1.2--B4, written in null coordinates
$g=2\lambda(x,y)\dd x\dd y$, recall the initial data
\eqref{eq:initial-data}, where $h(0)\neq0$. If $H\not\equiv0$, write
\begin{equation}\label{eq:H-order}
H(y)=y^\ell\widetilde H(y),
\qquad \widetilde H(0)\neq0.
\end{equation}
If $H\equiv0$, we set $\ell=\infty$; in this case $\widetilde H$ is not
used.

For B1.2--B4, write
\begin{equation}\label{eq:f-order}
m=\ord_0(f^2),\qquad
f(y)^2=y^m\widetilde{f^2}(y),\qquad
\widetilde{f^2}(0)\neq0.
\end{equation}
The value of $m$ for each family is recorded in the last column of
Table~\ref{tab:B-families}.

For B1.2, B1.3, and B2 in the critical case $m=2\ell$, put
\begin{equation}\label{eq:D-critical}
D(y)=\widetilde H(y)^2-\frac{f(y)^2}{y^{2\ell}}h(y)^2.
\end{equation}
The quotient $f(y)^2/y^{2\ell}$ is analytic and nonzero at the origin.
Thus $H(y)^2-f(y)^2h(y)^2=y^{2\ell}D(y).$
If $D\not\equiv0$, write
\begin{equation}\label{eq:D-order}
D(y)=y^n\widetilde D(y),
\qquad
\widetilde D(0)\neq0.
\end{equation}

In the normal forms $O,P,S$, the vector $\boldsymbol{\mathsf{c}}$ consists of the
coefficients of the polynomial part $\Pi_r(\mathsf{y})$ obtained in
\eqref{eq:trace-Taylor}. In the table below, $k$ and $s$ are the parameters
specified in Table~\ref{tab:B-families}, while $\ell$ is defined by
\eqref{eq:H-order}. The function $D$ occurs only in the critical case and is
defined by \eqref{eq:D-critical}; if $D\not\equiv0$, its order $n$ is defined
by \eqref{eq:D-order}.

\begin{theorem}
\label{thm:main}
Let $(g,F)$ be one of the generic normal forms R1--R3 displayed in
\eqref{eq:R-forms}, or one of the $\operatorname{gl}$-regular singular
normal forms B1.1--B4 given by \eqref{eq:B11},
\eqref{eq:B12-B2}, and \eqref{eq:B3-B4}, with the choices from
Table~\ref{tab:B-families}.
Let ${L}$ be defined by
\eqref{eq:L-definition}. Then the normal form of $L$ (up to adding $\lambda \operatorname{Id}$, $\lambda \in \mathbb{R}$) is given in
Table~\ref{tab:main-correspondence}. All normal forms occurring in the
third column are displayed in \eqref{eq:G-forms}, \eqref{eq:LMN}, and
Table~\ref{tab:OPS-series}.
\end{theorem}

\begingroup
\small
\renewcommand{\arraystretch}{1.15}
\setlength{\tabcolsep}{3pt}
\begin{longtable}{
    >{\centering\arraybackslash}p{0.15\textwidth}|
    p{0.34\textwidth}|
    p{0.41\textwidth}}
\caption{}
\label{tab:main-correspondence}\\

\hline
\textbf{Input: $(g,F)$}
&
\centering\arraybackslash\textbf{Subcase}
&
\centering\arraybackslash\textbf{Output: normal form of $L$}
\\
\hline
\endfirsthead

\hline
\textbf{Input: $(g,F)$}
&
\centering\arraybackslash\textbf{Subcase}
&
\centering\arraybackslash\textbf{Output: normal form of $L$}
\\
\hline
\endhead

\hline
\endfoot

$\mathrm{R1}$; \eqref{eq:R-forms}
& --
& $\mathrm{G1}$;  \eqref{eq:G-forms}\\ \hline
$\mathrm{R2}$;  \eqref{eq:R-forms}
& --
& $\mathrm{G2}$;  \eqref{eq:G-forms}\\ \hline
$\mathrm{R3}$; \eqref{eq:R-forms}
& \begin{tabular}[t]{@{}l@{}}
$Y(0)\neq0$\\
$Y\equiv0$\\
$Y=y^{2q-1}\widetilde Y$, $\widetilde Y(0)\neq0$\\
$Y=y^{2q}\widetilde Y$, $\widetilde Y(0)\neq0$
\end{tabular}
& \begin{tabular}[t]{@{}l@{}}
$\mathrm{G3}$; \eqref{eq:G-forms}\\
$L_{\mathrm{nil}}$; \eqref{eq:LMN}\\
$N_{2q-1}$; \eqref{eq:LMN}\\
$N_{2q}^{-\sgn\widetilde Y(0)}$; \eqref{eq:LMN}
\end{tabular}\\ \hline
$\mathrm{B1.1}$; see \eqref{eq:B11}
& --
& $L_{\mathrm{nd}}$; \eqref{eq:LMN}\\ \hline
$\mathrm{B1.2}$; \eqref{eq:B12-B2}, \eqref{eq:H-order},
\eqref{eq:D-critical}, \eqref{eq:D-order};
Table~\ref{tab:B-families}
& \begin{tabular}[t]{@{}l@{}}
$\ell>1$\\
$\ell=1$, $D(0)\neq0$\\
$\ell=1$, $D=y^n\widetilde D$, $n\geq1$\\
$\ell=1$, $D\equiv0$
\end{tabular}
& \begin{tabular}[t]{@{}l@{}}
$S_{\boldsymbol{\mathsf{c}}}^{2,1}$; Table~\ref{tab:OPS-series}\\
$S_{\boldsymbol{\mathsf{c}}}^{2,-\sgn D(0)}$; Table~\ref{tab:OPS-series}\\
$O_{1,\boldsymbol{\mathsf{c}}}^{2+n,\varepsilon}$; Table~\ref{tab:OPS-series}\\
$M_1$; \eqref{eq:LMN}
\end{tabular}\\ \hline
$\mathrm{B1.3}$; \eqref{eq:B12-B2}, \eqref{eq:H-order},
\eqref{eq:D-critical}, \eqref{eq:D-order};
Table~\ref{tab:B-families}
& \begin{tabular}[t]{@{}l@{}}
$s=2q<2\ell$\\
$s=2q+1<2\ell$\\
$s>2\ell$\\
$s=2\ell$, $D(0)\neq0$\\
$s=2\ell$, $D=y^n\widetilde D$, $n\geq1$\\
$s=2\ell$, $D\equiv0$, $\ell=2q-1$\\
$s=2\ell$, $D\equiv0$, $\ell=2q$
\end{tabular}
& \begin{tabular}[t]{@{}l@{}}
$S_{\boldsymbol{\mathsf{c}}}^{2q,1}$; Table~\ref{tab:OPS-series}\\
$S_{\boldsymbol{\mathsf{c}}}^{2q+1}$; Table~\ref{tab:OPS-series}\\
$P_{s-1,\boldsymbol{\mathsf{c}}}^{\ell,-\sgn\widetilde H(0)}$; Table~\ref{tab:OPS-series}\\
$S_{\boldsymbol{\mathsf{c}}}^{2\ell,-\sgn D(0)}$; Table~\ref{tab:OPS-series}\\
$O_{\ell,\boldsymbol{\mathsf{c}}}^{2\ell+n,\varepsilon}$; Table~\ref{tab:OPS-series}\\
$M_{2q-1}$; \eqref{eq:LMN}\\
$M_{2q}^{-\sgn\widetilde H(0)}$; \eqref{eq:LMN}
\end{tabular}\\ \hline
$\mathrm{B2}$; \eqref{eq:B12-B2}, \eqref{eq:H-order},
\eqref{eq:D-critical}, \eqref{eq:D-order};
Table~\ref{tab:B-families}
& \begin{tabular}[t]{@{}l@{}}
$k<\ell$\\
$k>\ell$\\
$k=\ell$, $D(0)\neq0$\\
$k=\ell$, $D=y^n\widetilde D$, $n\geq1$\\
$k=\ell$, $D\equiv0$, $k=2q-1$\\
$k=\ell$, $D\equiv0$, $k=2q$
\end{tabular}
& \begin{tabular}[t]{@{}l@{}}
$S_{\boldsymbol{\mathsf{c}}}^{2k,1}$; Table~\ref{tab:OPS-series}\\
$P_{2k-1,\boldsymbol{\mathsf{c}}}^{\ell,-\sgn\widetilde H(0)}$; Table~\ref{tab:OPS-series}\\
$S_{\boldsymbol{\mathsf{c}}}^{2k,-\sgn D(0)}$; Table~\ref{tab:OPS-series}\\
$O_{k,\boldsymbol{\mathsf{c}}}^{2k+n,\varepsilon}$; Table~\ref{tab:OPS-series}\\
$M_{2q-1}$; \eqref{eq:LMN}\\
$M_{2q}^{-\sgn\widetilde H(0)}$; \eqref{eq:LMN}
\end{tabular}\\ \hline
$\mathrm{B3.1}$; \eqref{eq:B3-B4};
Table~\ref{tab:B-families}
& --
& $S_{\boldsymbol{\mathsf{c}}}^{2,-1}$; Table~\ref{tab:OPS-series}\\ \hline
$\mathrm{B3.2}$; \eqref{eq:B3-B4}; \eqref{eq:H-order};
Table~\ref{tab:B-families}
& \begin{tabular}[t]{@{}l@{}}$k\leq\ell$\\ $k>\ell$\end{tabular}
& \begin{tabular}[t]{@{}l@{}}
$S_{\boldsymbol{\mathsf{c}}}^{2k,-1}$; Table~\ref{tab:OPS-series}\\
$P_{2k-1,\boldsymbol{\mathsf{c}}}^{\ell,\sgn\widetilde H(0)}$; Table~\ref{tab:OPS-series}
\end{tabular}\\ \hline
$\mathrm{B4}$; \eqref{eq:B3-B4}; \eqref{eq:H-order};
Table~\ref{tab:B-families}
& \begin{tabular}[t]{@{}l@{}}$k\leq\ell$\\ $k>\ell$\end{tabular}
& \begin{tabular}[t]{@{}l@{}}
$S_{\boldsymbol{\mathsf{c}}}^{2k,-1}$; Table~\ref{tab:OPS-series}\\
$P_{2k-1,\boldsymbol{\mathsf{c}}}^{\ell,\sgn\widetilde H(0)}$; Table~\ref{tab:OPS-series}
\end{tabular}\\ \hline
\end{longtable}
\endgroup

In all rows containing $O$, the sign is
\[
\varepsilon=
\begin{cases}
-\sgn \widetilde D(0),&n\text{ is even},\\
1,&n\text{ is odd}.
\end{cases}
\]

\section{Proof}
For R1--R3, direct substitution into \eqref{eq:L-definition} gives types G1, G2, and G3, respectively. In R3, the cases
$Y\equiv0$ and $Y(y)=y^j\widetilde Y(y),\ \widetilde Y(0)\neq0,$
give $L_{\mathrm{nil}}$ and, after a one-dimensional normalization,
\[
N_{2q-1}\quad (j=2q-1),\qquad
N_{2q}^{-\sgn\widetilde Y(0)}\quad (j=2q).
\]
For B1.1, the functions $\tr L$ and $\det L$ are functionally
independent, so $L$ is of type $L_{\mathrm{nd}}$.

For B1.2, B1.3, and B2, put
$\rho_+=\rho(x,y)$, $\rho_-=\rho(-x,y)$, and $K=fh$.
A direct calculation gives
\begin{equation}\label{eq:real-recognition}
\begin{aligned}
v&=-K(\rho_+)-H(\rho_+)+K(\rho_-)-H(\rho_-),\\
u&=(K(\rho_+)+H(\rho_+))(K(\rho_-)-H(\rho_-)),\\
J&=
2\bigl(K(\rho_+)+H(\rho_+)+K(\rho_-)-H(\rho_-)\bigr)
\bigl(K'(\rho_+)+H'(\rho_+)\bigr)\\
&\hspace{30mm}\cdot
\bigl(H'(\rho_-)-K'(\rho_-)\bigr)
\frac{f(\rho_+)f(\rho_-)}{f(y)}.
\end{aligned}
\end{equation}
In particular, $v(0,y)=-2H(y),\
u(0,y)=f(y)^2h(y)^2-H(y)^2.$ Since $\rho(x,0)=0$, one has
\[
\rho_+=y\widetilde\rho_+(x,y),\qquad
\rho_-=y\widetilde\rho_-(x,y),
\qquad \widetilde\rho_\pm(0,0)=1.
\]
If $a=\ord_0(H+K)$ and $b=\ord_0(H-K)$,
then \eqref{eq:real-recognition} gives
\[
d=a+b,\qquad N=d+m-2,\qquad r=m-1.
\]
The remaining recognition data are
\[
\begin{array}{c|c|c|c}
\text{condition}&d&\Pi_r(\mathsf{y})&\text{normal form}\\ \hline
m<2\ell,\ m=2q
&2q&\mathsf{y}^qC(\mathsf{y})
&S_{\boldsymbol{\mathsf{c}}}^{2q,1}\\[2pt]
m<2\ell,\ m=2q+1
&2q+1&\mathsf{y}^{q+1}C(\mathsf{y})
&S_{\boldsymbol{\mathsf{c}}}^{2q+1}\\[2pt]
m>2\ell
&2\ell&\mathsf{y}^{m-\ell}C(\mathsf{y})+2\varepsilon_0\mathsf{y}^\ell
&P_{m-1,\boldsymbol{\mathsf{c}}}^{\ell,\varepsilon_0}\\[2pt]
m=2\ell,\ D(0)\neq0
&2\ell&\mathsf{y}^\ell C(\mathsf{y})
&S_{\boldsymbol{\mathsf{c}}}^{2\ell,-\sgn D(0)}\\[2pt]
m=2\ell,\ D=y^n\widetilde D
&2\ell+n&\mathsf{y}^\ell C(\mathsf{y})
&O_{\ell,\boldsymbol{\mathsf{c}}}^{2\ell+n,\varepsilon},
\end{array}
\]
where
\[
\varepsilon_0=-\sgn \widetilde H(0),\qquad
\varepsilon=
\begin{cases}
-\sgn \widetilde D(0),&n\text{ is even},\\
1,&n\text{ is odd}.
\end{cases}
\]
Indeed, these values follow immediately by comparing the two terms in
\[
u(0,y)=f(y)^2h(y)^2-H(y)^2
\]
and using $v(0,y)=-2H(y)$. In the case $m>2\ell$, the normalization
$u=-\mathsf{y}^{2\ell}$ gives
\[
v(0,\mathsf{y})=-2\sgn \widetilde H(0)\mathsf{y}^\ell+O(\mathsf{y}^{\ell+1}),
\]
which determines the sign $\varepsilon_0$.

If $m=2\ell$ and $D\equiv0$, then $H=\tau fh$ for some
$\tau\in\{\pm1\}$. Hence one of $X,Y$ vanishes identically and
$u\equiv0$. A one-dimensional normalization of $v$ gives
\[
M_{2q-1}\quad\text{if }\ell=2q-1,\qquad
M_{2q}^{-\sgn\widetilde H(0)}
\quad\text{if }\ell=2q.
\]

For B3.1, B3.2, and B4, put
$\rho_+=\rho(\ii x,y)$, $\rho_-=\rho(-\ii x,y)$,
$K=fh$, and $W=H+\ii K$. A direct calculation gives
\begin{equation}\label{eq:complex-recognition}
\begin{aligned}
v&=W(\rho_+)+\overline W(\rho_-),\qquad
u=-W(\rho_+)\overline W(\rho_-),\\
J&=
2\ii\bigl(\overline W(\rho_-)-W(\rho_+)\bigr)
W'(\rho_+)\overline W'(\rho_-)
\frac{f(\rho_+)f(\rho_-)}{f(y)}.
\end{aligned}
\end{equation}
At $x=0$,
\[
v(0,y)=2H(y),\qquad
u(0,y)=-H(y)^2-f(y)^2h(y)^2.
\]
The two terms in $u(0,y)$ cannot cancel; hence
$d=\min\{m,2\ell\}$, $N=m+d-2$, and $r=m-1$.

The remaining data are
\[
\begin{array}{c|c|c|c}
\text{condition}&d&\Pi_r(\mathsf{y})&\text{normal form}\\ \hline
m\leq2\ell
&m&\mathsf{y}^{m/2}C(\mathsf{y})
&S_{\boldsymbol{\mathsf{c}}}^{m,-1}\\[2pt]
m>2\ell
&2\ell&\mathsf{y}^{m-\ell}C(\mathsf{y})+2\varepsilon_0\mathsf{y}^\ell
&P_{m-1,\boldsymbol{\mathsf{c}}}^{\ell,\varepsilon_0},
\end{array}
\qquad
\varepsilon_0=\sgn \widetilde H(0).
\]
This gives $S_{\boldsymbol{\mathsf{c}}}^{2,-1}$ for B3.1 and, writing $m=2k$,
\[
S_{\boldsymbol{\mathsf{c}}}^{2k,-1}\quad(k\leq\ell),\qquad
P_{2k-1,\boldsymbol{\mathsf{c}}}^{\ell,\sgn\widetilde H(0)}
\quad(k>\ell)
\]
for B3.2 and B4.

Applying the construction
\eqref{eq:w-general}--\eqref{eq:normalized-vu} to the data above gives
the normal forms stated in Table~\ref{tab:main-correspondence}. Formula
\eqref{eq:reconstruction} identifies the operators wherever
$\dd v\wedge\dd u\neq0$, and the equality extends to the origin by
analyticity.

\textbf{Acknowledgment.} The author thanks Alexey Bolsinov, Vladimir Matveev and Matteo Manenti for their comments on the manuscript. The author was supported by the DFG (project No. 529233771).

\end{document}